\documentclass[submission]{FPSAC2020}

\usepackage[utf8x]{inputenc}
\usepackage[english]{babel}
\usepackage{amsmath,amsfonts,amssymb,amsthm}
\usepackage[T1]{fontenc}
\usepackage{hyperref}
\usepackage{cite}
\usepackage{paralist}
\usepackage{multicol}
\usepackage{stmaryrd}
\usepackage{subfig}
\usepackage{multirow}
\usepackage{tikz}
\usepackage{tikz-cd}
\usetikzlibrary{fit}
\usetikzlibrary{arrows.meta}

\usepackage{todonotes}
\usepackage{xcolor}
\definecolor{ColBlack}{RGB}{0,0,0} % Black.
\definecolor{ColWhite}{RGB}{255,255,255} % White.

\definecolor{ColA}{RGB}{171,57,39} % Red.
\definecolor{ColB}{RGB}{44,171,70} % Green.

\newtheorem{theorem}{Theorem}[section]
\newtheorem{proposition}[theorem]{Proposition}
\newtheorem{lemma}[theorem]{Lemma}
\theoremstyle{definition}
\newtheorem{definition}{Definition}

\numberwithin{equation}{section}

\renewcommand{\leq}{\leqslant}
\renewcommand{\geq}{\geqslant}

\title{Graph insertion operads}
\author{
    J.-C. Aval%
    \thanks{\href{mailto:aval@labri.fr}
        {\tt aval@labri.fr}}%
    \addressmark{1},
    S. Giraudo%
    \thanks{\href{mailto:samuele.giraudo@u-pem.fr}
        {\tt samuele.giraudo@u-pem.fr}}%
    \addressmark{2},
    T. Karaboghossian%
    \thanks{\href{mailto:theo.karaboghossian@u-bordeaux.fr}
        {\tt theo.karaboghossian@u-bordeaux.fr}}%
    \addressmark{1},
    \and
    A. Tanasa%'
    \thanks{\href{mailto:ntanasa@u-bordeaux.fr}
        {\tt ntanasa@u-bordeaux.fr}}%
    \addressmark{1}\addressmark{3}
}
\address{
    \addressmark{1} LaBRI (UMR CNRS $5800$), Univ. Bordeaux, $33405$ Talence, France.
    \\
    \addressmark{2} LIGM, Univ. Gustave Eiffel, CNRS, ESIEE Paris, F-$77454$
    Marne-la-Vallée, France.
    \\
    \addressmark{3} IUF Paris, France and H. Hulubei Nat. Inst. Phys. Nucl. Engineering
    Magurele, Romania.
}
\abstract{Using the combinatorial species setting, we propose two new operad structures on
multigraphs and on pointed oriented multigraphs. The former can be considered as a canonical
operad on multigraphs, directly generalizing the Kontsevich-Willwacher operad, and has many
interesting suboperads. The latter is a natural extension of the pre-Lie operad in a sense
developed here and related to the multigraph operad. We study some of the finitely generated
suboperads of the multigraph operad and establish links between them and the commutative
operad and the commutative magmatic operad.}
\received{\today}

\tikzstyle{Centering}=[{baseline={([yshift=-0.5ex]current bounding box.center)}}]

\tikzstyle{NodeGraph}=[circle,draw=ColA!90,fill=ColA!10,inner sep=1pt,minimum size=4mm,
thick,font=\scriptsize]
\tikzstyle{UnlabeledNodeGraph}=[NodeGraph,minimum size=2mm]
\tikzstyle{RootGraph}=[NodeGraph,rectangle]
\tikzstyle{EdgeGraph}=[ColB!70,cap=round,very thick]
\tikzstyle{ArcGraph}=[EdgeGraph,->]

\newcommand{\Hide}[1]{\tt HIDEN}

\newcommand{\OEIS}[1]{\href{http://oeis.org/#1}{{\bf #1}}}

\newcommand{\K}{\mathbb{K}}

\newcommand{\Identity}{\mathcal{I}}
\newcommand{\Operad}{\mathcal{O}}

\newcommand{\Ope}{\mathrm{Ope}}
\newcommand{\FreeOp}{\mathbf{Free}}

\newcommand{\Com}{\mathbf{Com}}
\newcommand{\PLie}{\mathbf{PLie}}
\newcommand{\NAP}{\mathbf{NAP}}
\newcommand{\ComMag}{\mathbf{ComMag}}

\newcommand{\MG}{\mathbf{MG}}
\newcommand{\G}{\mathbf{G}}

\newcommand{\T}{\mathbf{T}}
\newcommand{\ST}{\mathbf{ST}}
\newcommand{\SP}{\mathbf{SP}}
\newcommand{\LP}{\mathbf{LP}}

\newcommand{\Points}[2]{
    \begin{tikzpicture}[Centering,scale=.6]
        \node[NodeGraph](a)at(0,0){$#1$};
        \node[NodeGraph](b)at(1,0){$#2$};
    \end{tikzpicture}}

\newcommand{\Segment}[2]{
    \begin{tikzpicture}[Centering,scale=.6]
        \node[NodeGraph](a)at(0,0){$#1$};
        \node[NodeGraph](b)at(1,0){$#2$};
        \draw[EdgeGraph](a)--(b);
    \end{tikzpicture}}

\begin{document}
\maketitle

%%%%%%%%%%%%%%%%%%%%%%%%%%%%%%%%%%%%%%%%%%%%%%%%%%%%%%%%%%%%%%%%%%%%%%%%%%%%%%%%%%%%%%%%%%%%
%%%%%%%%%%%%%%%%%%%%%%%%%%%%%%%%%%%%%%%%%%%%%%%%%%%%%%%%%%%%%%%%%%%%%%%%%%%%%%%%%%%%%%%%%%%%
%%%%%%%%%%%%%%%%%%%%%%%%%%%%%%%%%%%%%%%%%%%%%%%%%%%%%%%%%%%%%%%%%%%%%%%%%%%%%%%%%%%%%%%%%%%%
\section*{Introduction}
Operads are mathematical structures which have been intensively studied in the context of
topology, algebra \cite{LV12} but also of combinatorics~\cite{CSLC} ---see for
example~\cite{Men15, Gir18} for general references on symmetric and non-symmetric operads,
set-operads through species, {\em etc.} In the last decades, several interesting operads on
trees have been defined. Amongst these tree operads, maybe the most studied are the pre-Lie
operad $\PLie$~\cite{CL01} and the nonassociative permutative operad $\NAP$~\cite{Liv06}.

However, it seems to us that a natural question to ask is what kind of operads can be
defined on graphs and what are their properties? The need for defining appropriate graph
operads comes from combinatorics, where graphs are, just like trees, natural objects to
study, but also from physics, where it was recently proposed to use graph operads in order
to encode the combinatorics of the renormalization of Feyman graphs in quantum field
theory~\cite{Kreimer:2000ja}.

Other graph operads have been defined for example in~\cite{Kon99,Wil05,Men15,Gir17,MV19}.
In this paper, we go further in this direction and we define, using the combinatorial
species~\cite{BLL98} setting, new graph operads. Moreover, we investigate several properties
of these operads: we describe an explicit link with the pre-Lie tree operad mentioned above,
and we study interesting (finitely generated) suboperads.

This paper is organized as follows. Section~\ref{sec:preliminaries} contains elementary
definitions of species and operads. In Section~\ref{sec:graph_operads} we define and study
the main operads of interest of this paper. Section~\ref{sec:suboperads} is devoted to the
study of finitely generated suboperads.

This text is an extended abstract. The long version of this work~\cite{AGKT19} contains a
more general definition of graph insertion operads as well as new operad constructions and
all the proofs of the results presented here.

%%%%%%%%%%%%%%%%%%%%%%%%%%%%%%%%%%%%%%%%%%%%%%%%%%%%%%%%%%%%%%%%%%%%%%%%%%%%%%%%%%%%%%%%%%%%
%%%%%%%%%%%%%%%%%%%%%%%%%%%%%%%%%%%%%%%%%%%%%%%%%%%%%%%%%%%%%%%%%%%%%%%%%%%%%%%%%%%%%%%%%%%%
%%%%%%%%%%%%%%%%%%%%%%%%%%%%%%%%%%%%%%%%%%%%%%%%%%%%%%%%%%%%%%%%%%%%%%%%%%%%%%%%%%%%%%%%%%%%
\section{Species, operads and graphs} \label{sec:preliminaries}
Most definitions, results and proofs of this section can be found with more details
in~\cite{Men15}. We refer the reader to~\cite{BLL98} for the theory of species and
to~\cite{LV12} for the theory of operads. In all the following, $\K$ is a field of
characteristic zero. For any positive integer $n$, $[n]$ stands for the set $\{1,\dots,n\}$.

%%%%%%%%%%%%%%%%%%%%%%%%%%%%%%%%%%%%%%%%%%%%%%%%%%%%%%%%%%%%%%%%%%%%%%%%%%%%%%%%%%%%%%%%%%%%
%%%%%%%%%%%%%%%%%%%%%%%%%%%%%%%%%%%%%%%%%%%%%%%%%%%%%%%%%%%%%%%%%%%%%%%%%%%%%%%%%%%%%%%%%%%%
\begin{definition}
A \textit{set species} $S$ consists of the following data. For each finite set $V$ a set
$S[V]$, and for each bijection of finite sets $\sigma: V\rightarrow V'$ a map
$S[\sigma]:S[V]\rightarrow S[V']$.  These maps should be such that $S[\sigma\circ\tau] =
S[\sigma]\circ S[\tau]$ and $S[\Identity] = \Identity$, where $\Identity$ is the identity
map.

A \textit{morphism of set species} $f: R\rightarrow S$ is a collection of maps $f_V : R[V]
\rightarrow S[V]$ such that for each bijection $\sigma: V\rightarrow V'$, $f_{V'}\circ
R[\sigma] = S[\sigma]\circ f_V$.  A set species $S$ is \textit{positive} if $|S[\emptyset]|
= 0$ and \textit{connected} if $|S[\{v\}]| = 1$ for any singleton $\{v\}$.
\end{definition}

Switching sets with vector spaces, maps to linear maps and cardinality to dimension in the
previous definition, we obtain the definitions of \textit{linear species}, \textit{morphisms
of linear species}, \textit{positive linear species}, and \textit{connected linear species}.
The \textit{Hilbert series} of a linear species $S$ is the formal series $\mathcal{H}_S(t) =
\sum_{n \geq 0} \dim S[[n]] \frac{x^n}{n!}$.  For $S$ a set species, we denote by $\K S$ the
linear species defined by  $(\K S)[V] = \K S[V]$, where $\K S[V]$ is the $\K$-linear span of
$S[V]$. The linear space $\K S[V]$ is naturally equipped with a scalar product
$(-|-)_{S[V]}$ by setting that $S[V]$ is an orthonormal basis. The \textit{support} of $x
\in \K S[V]$ is the set $\{y\in S[V] : (x| y)_{S[V]} > 0\}$.

In all the following, $V$ always denotes a finite set.  Let $R$ and $S$ be two linear
species. We recall the classical constructions on species: $(R+S)[V] = R[V]\oplus S[V]$
(\textit{sum}), $(R \cdot S)[V] = \bigoplus_{V_1\sqcup V_2 = V} R[V_1]\otimes S[V_2]$
(\textit{product}), $(R\times S)[V] = R[V]\otimes S[V]$ (\textit{Hadamard product}), $R'[V]
= R[V \sqcup\{\ast\}]$ (\textit{derivative}), $R^{\bullet}[V] = R[V]\times V$
(\textit{pointing}), and $E(R)[V]= \bigoplus_{\cong}\bigotimes_{W\in V/\cong} R[W]$
(\textit{assembly}) where $\cong$ run over the set of the equivalence relations on $V$.
These definitions are compatible with the functor $S\mapsto \K S$, {\em e.g.}, $\K(R+S) = \K
R \oplus \K S$.

Let $X$ be the set species defined by $X[\{v\}] = \{v\}$ and $X[V] = \emptyset$ if $V$ is
not a singleton.

%%%%%%%%%%%%%%%%%%%%%%%%%%%%%%%%%%%%%%%%%%%%%%%%%%%%%%%%%%%%%%%%%%%%%%%%%%%%%%%%%%%%%%%%%%%%
%%%%%%%%%%%%%%%%%%%%%%%%%%%%%%%%%%%%%%%%%%%%%%%%%%%%%%%%%%%%%%%%%%%%%%%%%%%%%%%%%%%%%%%%%%%%
\begin{definition} \label{def:operads}
A \textit{(symmetric) set} (resp. \textit{linear}) \textit{operad} is a positive set (resp.
linear) species $\Operad$ together with a \textit{unit} $e: X\rightarrow \Operad$ (resp. $e
: \K X \rightarrow \Operad$) and a \textit{partial composition map}
$\circ_{\ast}:\Operad'\cdot\Operad\rightarrow \Operad$, such that the following three
diagrams commute

\begin{minipage}{.26\textwidth}
\begin{equation} \footnotesize
    \begin{tikzcd}
        \Operad'' \cdot \Operad^2 \arrow[r, "\circ_{\ast_1}"]
        \arrow[d,"\circ_{\ast_2}\circ\Identity\cdot\tau"]
        & \Operad'\cdot\Operad \arrow[d, "\circ_{\ast_2}"]
        \\
        \Operad'\cdot \Operad \arrow[r, "\circ_{\ast_1}"] & \Operad
    \end{tikzcd}
\end{equation}
\end{minipage}
\begin{minipage}{.28\textwidth}
\begin{equation} \footnotesize
    \begin{tikzcd}
        \Operad'\cdot \Operad'\cdot\Operad \arrow[r, "\circ_{\ast_1}\cdot\Identity"]
        \arrow[d, "\Identity\cdot\circ_{\ast_2}"] & \Operad'\cdot\Operad
        \arrow[d, "\circ_{\ast_2}"]
        \\
        \Operad'\cdot \Operad \arrow[r, "\circ_{\ast_1}"] & \Operad
    \end{tikzcd}
\end{equation}
\end{minipage}
\begin{minipage}{.36\textwidth}
\begin{equation} \footnotesize
    \begin{tikzcd}
        \Operad'\cdot \K X \arrow[r, "\Operad'\cdot e"] \arrow[rd, "p"]
        & \Operad'\cdot\Operad \arrow[d, "\circ_{\ast}"] & \K X'\cdot \Operad
        \arrow[l, "e'\cdot \Operad" swap]\arrow[dl, "\cong" swap]
        \\
        & \Operad 
    \end{tikzcd}
\end{equation}
\end{minipage}

\noindent where $\tau_V:x\otimes y\in \Operad^2[V] \mapsto y\otimes x\in \Operad^2[V]$, and
$p_V: x\otimes v \mapsto \Operad[\ast\mapsto v](x)$ with $\ast\mapsto v$ the bijection that
sends $\ast$ on $v$ and is the identity on $V\setminus\{v\}$.

An \textit{operad morphism} is a species morphism compatible with the units and the
partial composition maps.
\end{definition}

Remark that if $(S,e,\circ_{\ast})$ is a set-operad, then extending $e$ and $\circ_{\ast}$
linearly makes $(\K S,e,\circ_{\ast})$ a linear operad. In all the following, $e$ will often
be trivial and we will not mention it. From now on all the considered species will be
positive. Except for the set operad $\ComMag$ (see Section~\ref{sec:suboperads}), we will
only consider linear operad, hence we will write species and operad for linear species and
linear opeard.

An \textit{ideal} of an operad $\Operad$ is a subspecies $S$ such that the image of the
products $\Operad'\cdot S$ and $S'\cdot\Operad$ by the partial composition maps are in $S$.
The \textit{quotient species} $\Operad/ S$ defined by $(\Operad/S)[V] = \Operad[V]/S[V]$ is
then an operad with the natural partial composition and unit.

We now need to recall the notion of free operad. For this we first introduce some notations.
For $V$ a set, let $\mathcal{T}$ be the species of \textit{trees} defined as follows. For
any set $V$, $\mathcal{T}[V]$ is the set such that
\begin{itemize}
    \item if $V=\{v\}$ is a singleton, then the sole element of $\mathcal{T}[V]$ is the tree
    reduced to a leaf labelled by $\{v\}$.
    \item Otherwise, let $\pi =(\pi_1,\dots,\pi_k)$ be a partition of $V$ and
    $t_1,\dots,t_k$ be respectively elements of $\mathcal{T}[\pi_1]$, \dots,
    $\mathcal{T}[\pi_k]$. Then the tree consisting in an internal node having from left to
    right $t_1$, \dots, $t_k$ as sub-trees is an element of $\mathcal{T}[V]$.
\end{itemize}

Let now $G$ be a positive species. The \textit{free operad} $\FreeOp_G$ over $G$ is defined
as follows. As a species, $\FreeOp_G$ is such that for any set $V$, $\FreeOp_G[V]$ is the
set of labeled versions of the trees of $\mathcal{T}[V]$: any internal nodes having $k$
children of a tree is labeled by an element of $G[[k]]$.  The partial composition of
$\FreeOp_G$, denoted by $\circ^{\xi}$ is the grafting of trees: for any disjoint sets $V_1$
and $V_2$ with $\ast\in V_1$, and $t_1 \in \FreeOp_G[V_1]$ and $t_2 \in \FreeOp_G[V_2]$,
$t_1\circ_{\ast}^{\xi} t_2$ is the tree obtained by grafting $t_2$ on the leaf $\ast$ of
$t_1$. Moreover, for any $k \geq 0$, we denote by $\FreeOp_G^{(k)}$ the subspecies of
$\FreeOp_G$ of trees with $k$ exactly internal nodes.

If $R$ is a subspecies of $\FreeOp_G$, we denote by $(R)$ the smallest ideal containing $R$
and write that $(R)$ is \textit{generated by $R$}.

For any species $S$ we denote by $S^{\vee}$ the species defined by $S^{\vee}[V] = S[V]^*$
and $S^{\vee}[\sigma](x) = \text{sign}(\sigma)x\circ S[\sigma^{-1}]$.

\begin{definition}
Let $G$ be a positive species and $R$ be a subspecies of $\FreeOp_G$. Let $\Ope(G,R) =
\FreeOp_G/(R)$. The operad $\Ope(G,R)$ is \textit{binary} if the species $G$ of generators
is concentrated in cardinality $2$ ({\em i.e.}, for all $n \ne 2$, $G[[n]] = \{0\}$). This
operad is \textit{quadratic} if $R$ is a subspecies of $\FreeOp_G^{(2)}$.
\end{definition}

\begin{definition}
Let $\Operad =\Ope(G,R)$ be a binary quadratic operad. Let us define the linear form
$\langle -, -\rangle$ on $\FreeOp_{G^{\vee}}^{(2)}\times\FreeOp_{G}^{(2)}$ as follows.  For
$V=\{a,b,c\}, f_1\in {G^{\vee}}'[\{a\}], f_2\in G[\{b,c\}], x_1\in G'[\{a\}]$, and $x_2\in
G[\{b,c\}]$:
\begin{equation}
    \langle f_1\circ_{\ast} f_2 , x_1\circ_{\ast} x_2 \rangle = f_1(x_1)f_2(x_2).
\end{equation}
The \textit{Koszul dual} of $\Operad$ is then the operad $\Operad^!=\Ope(G^{\vee},R^{\bot})$
where $R^{\bot}$ is the orthogonal of $R$ for $\langle -,-\rangle$.
\end{definition}

When $\Operad$ is quadratic and its Koszul complex is acyclic~\cite{LV12}, $\Operad$ is
a \textit{Koszul operad}. In this case, the Hilbert series of $\Operad$ and of its Koszul
dual are related by the identity
\begin{equation}
    \mathcal{H}_{\Operad}(-\mathcal{H}_{\Operad^!}(-t)) = t.
\end{equation}

%%%%%%%%%%%%%%%%%%%%%%%%%%%%%%%%%%%%%%%%%%%%%%%%%%%%%%%%%%%%%%%%%%%%%%%%%%%%%%%%%%%%%%%%%%%%
%%%%%%%%%%%%%%%%%%%%%%%%%%%%%%%%%%%%%%%%%%%%%%%%%%%%%%%%%%%%%%%%%%%%%%%%%%%%%%%%%%%%%%%%%%%%
%%%%%%%%%%%%%%%%%%%%%%%%%%%%%%%%%%%%%%%%%%%%%%%%%%%%%%%%%%%%%%%%%%%%%%%%%%%%%%%%%%%%%%%%%%%%
\section{Graph operads} \label{sec:graph_operads}
A \textit{multigraph} on $V$ is a multiset of unordered pairs in $V^2$ which we call
\textit{edges}. In this context, the elements of $V$ are called vertices and the elements in
$V$ which are in no edge are called \textit{isolated vertices}. A multigraph on $V$ is
\textit{connected} if for every vertices $v$ and $v'$, there is a sequence of edges
$e_1,\dots,e_k$ such that $v\in e_1$, $v'\in e_k$ and $e_i\cap e_{i+1}$ for every $1\leq
i<k$.  A \textit{graph} on $V$ is a multigraph on $V$ which is a set and with no edge in
$\{\{v,v\} : v\in V\}$. We denote by $\MG$ the set species of multigraphs, by $\G$ its set
subspecies of graphs, and by $\MG_c$ and $\G_c$ their connected counterparts. We finally
denote by $\T$ the set subspecies of $\G_c$ restricted to trees.

Let $V_1$ and $V_2$ be two disjoint sets such that $\ast \in V_1$.  For any multigraphs $g_1
\in \MG[V_1]$ and $g_2 \in \MG[V_2]$, the \textit{insertion} of $g_2$ into $g_1$ is the sum
of all the multigraphs of $\MG[V_1 \setminus \{\ast\} \sqcup V_2]$ obtained by the following
process:
\begin{enumerate}
    \item Do the disjoint union of $g_1$ and $g_2$;
    \item Remove the vertex $\ast$. We then have some edges with one (or two if $\ast$ has
    loops) loose end(s);
    \item Connect each loose end to any vertex in $V_2$.
\end{enumerate}
For instance,
\begin{equation}\begin{split}
    \begin{tikzpicture}[Centering,scale=1]
        \tikzset{every loop/.style={}}
        \node[NodeGraph](a)at(0,0){$a$};
        \node[NodeGraph](s)at(1,0){$\ast$};
        \draw[EdgeGraph](a)edge[bend left=40](s);
        \draw[EdgeGraph](a)edge[bend right=40](s);
        \draw[EdgeGraph](s)edge[loop above](s);
        \draw[EdgeGraph,draw=ColWhite](s)edge[loop below](s);
    \end{tikzpicture}
    \enspace \circ_\ast \enspace
    \begin{tikzpicture}[Centering,scale=1]
        \tikzset{every loop/.style={}}
        \node[NodeGraph](b)at(0,0){$b$};
        \node[NodeGraph](c)at(1,0){$c$};
        \draw[EdgeGraph](b)--(c);
        \draw[EdgeGraph](c)edge[loop above](c);
        \draw[EdgeGraph,draw=ColWhite](s)edge[loop below](s);
    \end{tikzpicture}
    & \enspace = \enspace
    \begin{tikzpicture}[Centering,scale=1]
        \tikzset{every loop/.style={}}
        \node[NodeGraph](a)at(0,0){$a$};
        \node[NodeGraph](b)at(1,0){$b$};
        \node[NodeGraph](c)at(2,0){$c$}; 
        \draw[EdgeGraph](a)edge[bend left=40](b);
        \draw[EdgeGraph](a)edge[bend right=40](b);
        \draw[EdgeGraph](b)edge[loop below](b);
        \draw[EdgeGraph](b)--(c);
        \draw[EdgeGraph](c)edge[loop above](c);
    \end{tikzpicture}
    \enspace + \enspace
    \begin{tikzpicture}[Centering,scale=1]
        \tikzset{every loop/.style={}}
        \node[NodeGraph](a)at(0,0){$a$};
        \node[NodeGraph](b)at(1,0){$b$};
        \node[NodeGraph](c)at(2,0){$c$};
        \draw[EdgeGraph](a)edge[bend left=40](b);
        \draw[EdgeGraph](a)edge[bend right=40](b);
        \draw[EdgeGraph](c)edge[loop below](c);
        \draw[EdgeGraph](b)--(c);
        \draw[EdgeGraph](c)edge[loop above](c);
    \end{tikzpicture}
    \enspace + \enspace
    2\,
    \begin{tikzpicture}[Centering,scale=1]
        \tikzset{every loop/.style={}}
        \node[NodeGraph](a)at(0,0){$a$};
        \node[NodeGraph](b)at(1,0){$b$};
        \node[NodeGraph](c)at(2,0){$c$};
        \draw[EdgeGraph](a)edge[bend left=40](b);
        \draw[EdgeGraph](a)edge[bend right=40](b);
        \draw[EdgeGraph](b)edge[bend left=40](c);
        \draw[EdgeGraph](b)edge[bend right=40](c);
        \draw[EdgeGraph](c)edge[loop above](c);
        \draw[EdgeGraph,draw=ColWhite](c)edge[loop below](c);
    \end{tikzpicture}
    \\
    & \quad + \enspace
    2\,
    \begin{tikzpicture}[Centering,scale=1]
        \tikzset{every loop/.style={}}
        \node[NodeGraph](a)at(0,0){$a$};
        \node[NodeGraph](b)at(1,0){$b$};
        \node[NodeGraph](c)at(2,0){$c$};
        \draw[EdgeGraph](a)--(b);
        \draw[EdgeGraph](a)edge[bend right=40](c);
        \draw[EdgeGraph](b)edge[loop above](b);
        \draw[EdgeGraph](b)--(c);
        \draw[EdgeGraph](c)edge[loop above](c);
    \end{tikzpicture}
    \enspace + \enspace
    2\,
    \begin{tikzpicture}[Centering,scale=1]
        \tikzset{every loop/.style={}}
        \node[NodeGraph](a)at(0,0){$a$};
        \node[NodeGraph](b)at(1,0){$b$};
        \node[NodeGraph](c)at(2,0){$c$};
        \draw[EdgeGraph](a)--(b);
        \draw[EdgeGraph](a)edge[bend right=40](c);
        \draw[EdgeGraph](c)edge[loop below](c);
        \draw[EdgeGraph](b)--(c);
        \draw[EdgeGraph](c)edge[loop above](c);
    \end{tikzpicture}
    \enspace + \enspace
    4\,
    \begin{tikzpicture}[Centering,scale=1]
        \tikzset{every loop/.style={}}
        \node[NodeGraph](a)at(0,0){$a$};
        \node[NodeGraph](b)at(1,0){$b$};
        \node[NodeGraph](c)at(2,0){$c$};
        \draw[EdgeGraph](a)--(b);
        \draw[EdgeGraph](a)edge[bend right=40](c);
        \draw[EdgeGraph](b)edge[bend right=40](c);
        \draw[EdgeGraph](b)edge[bend left=40](c);
        \draw[EdgeGraph](c)edge[loop above](c);
    \end{tikzpicture}
    \\
    & \quad + \enspace
    \begin{tikzpicture}[Centering,scale=1]
        \tikzset{every loop/.style={}}
        \node[NodeGraph](a)at(0,0){$a$};
        \node[NodeGraph](b)at(1,0){$b$};
        \node[NodeGraph](c)at(2,0){$c$};
        \draw[EdgeGraph](a)edge[bend right=40](c);
        \draw[EdgeGraph](a)edge[bend left=40](c);
        \draw[EdgeGraph](b)edge[loop left](b);
        \draw[EdgeGraph](b)--(c);
        \draw[EdgeGraph](c)edge[loop above](c);
    \end{tikzpicture}
    \enspace + \enspace
    \begin{tikzpicture}[Centering,scale=1]
        \tikzset{every loop/.style={}}
        \node[NodeGraph](a)at(0,0){$a$};
        \node[NodeGraph](b)at(1,0){$b$};
        \node[NodeGraph](c)at(2,0){$c$};
        \draw[EdgeGraph](a)edge[bend right=40](c);
        \draw[EdgeGraph](a)edge[bend left=40](c);
        \draw[EdgeGraph](c)edge[loop below](c);
        \draw[EdgeGraph](b)--(c);
        \draw[EdgeGraph](c)edge[loop above](c);
    \end{tikzpicture}
    \enspace + \enspace
    2\,
    \begin{tikzpicture}[Centering,scale=1]
        \tikzset{every loop/.style={}}
        \node[NodeGraph](a)at(0,0){$a$};
        \node[NodeGraph](b)at(1,0){$b$};
        \node[NodeGraph](c)at(2,0){$c$};
        \draw[EdgeGraph](a)edge[bend right=40](c);
        \draw[EdgeGraph](a)edge[bend left=40](c);
        \draw[EdgeGraph](b)edge[bend left=20](c);
        \draw[EdgeGraph](b)edge[bend right=20](c);
        \draw[EdgeGraph](c)edge[loop above](c);
    \end{tikzpicture}.
\end{split}\end{equation}

\begin{theorem}
    The species $\K \MG$, endowed with the insertion as partial composition, is an operad.
\end{theorem}

We call $\K \MG$ the {\em graph insertion operad}. It is straightforward to observe that the
species $\K \G$ and $\K \MG_c$ are suboperads of $\K \MG$, that $\K \G_c$ a suboperad of $\K
\G$, and that $\K \T$ is a suboperad of $\K \G_c$. In particular, this structure on $\K \G$
is known as the Kontsevich-Willwacher operad~\cite{MV19}. For $\K \G$, the insertion
reformulates more formally as follows. For any $g_1 \in \G[V_1]$ and $g_2 \in \G[V_2]$ such
that $V_1$ and $V_2$ are two disjoint sets and $\ast \in V_1$,
\begin{equation}
    g_1 \circ_{\ast} g_2 =
    \sum_{f : N_\ast \to V_2}
    \left(g_1 \setminus \{\ast\}\right) \cup g_2 \cup \bigcup_{v\in N_\ast} \{v,f(v)\},
\end{equation}
where $N_\ast$ is the set of neighbours of $\ast$ in $g_1$. For instance,
\begin{equation}
    \begin{tikzpicture}[Centering,scale=.5]
        \node[NodeGraph](a)at(1,1){$a$};
        \node[NodeGraph](s)at(0,0){$\ast$};
        \node[NodeGraph](b)at(1,-1){$b$};
        \draw[EdgeGraph](a)--(s);
        \draw[EdgeGraph](s)--(b);
    \end{tikzpicture}
    \enspace \circ_\ast \enspace
    \begin{tikzpicture}[Centering,scale=.7]
        \node[NodeGraph](c)at(0,0){$c$};
        \node[NodeGraph](d)at(1,0){$d$};
        \draw[EdgeGraph](c)--(d);
    \end{tikzpicture}
    \enspace = \enspace
    \begin{tikzpicture}[Centering,scale=.5]
        \node[NodeGraph](a)at(1,1){$a$};
        \node[NodeGraph](b)at(1,-1){$b$};
        \node[NodeGraph](c)at(0,0){$c$};
        \node[NodeGraph](d)at(2,0){$d$};
        \draw[EdgeGraph](c)--(a);
        \draw[EdgeGraph](c)--(d);
        \draw[EdgeGraph](c)--(b);
    \end{tikzpicture}
    \enspace + \enspace
    \begin{tikzpicture}[Centering,scale=.5]
        \node[NodeGraph](a)at(1,1){$a$};
        \node[NodeGraph](b)at(1,-1){$b$};
        \node[NodeGraph](c)at(0,0){$c$};
        \node[NodeGraph](d)at(2,0){$d$};
        \draw[EdgeGraph](c)--(d);
        \draw[EdgeGraph](d)--(a);
        \draw[EdgeGraph](c)--(b);
    \end{tikzpicture}
    \enspace + \enspace
    \begin{tikzpicture}[Centering,scale=.5]
        \node[NodeGraph](a)at(1,1){$a$};
        \node[NodeGraph](b)at(1,-1){$b$};
        \node[NodeGraph](c)at(0,0){$c$};
        \node[NodeGraph](d)at(2,0){$d$};
        \draw[EdgeGraph](c)--(d);
        \draw[EdgeGraph](c)--(a);
        \draw[EdgeGraph](d)--(b);
    \end{tikzpicture}
    \enspace + \enspace
    \begin{tikzpicture}[Centering,scale=.5]
        \node[NodeGraph](a)at(1,1){$a$};
        \node[NodeGraph](b)at(1,-1){$b$};
        \node[NodeGraph](c)at(0,0){$c$};
        \node[NodeGraph](d)at(2,0){$d$};
        \draw[EdgeGraph](c)--(d);
        \draw[EdgeGraph](d)--(a);
        \draw[EdgeGraph](d)--(b);
    \end{tikzpicture}.
\end{equation}
It is easy to observe that all graphs appearing in $g_1 \circ_\ast g_2$ have $1$ as
coefficient.

While $\K \MG$ has an involved structure we will see that it has many interesting
suboperads. Let us start by giving some basic results on $\K \G$.

Let $S$ be a species, $I$ be a set, $\{V_i\}_{i\in I}$ be a family of finite sets, and
$x_i\in S[V_i]$ for all $i\in I$. We call {\em subspecies of $S$ generated by $\{x_i\}_{i\in
I}$} the smallest subspecies of $S$ containing the family $\{x_i\}_{i\in I}$.  If $S$ is
furthermore an operad, we call {\em suboperad of $S$ generated by $\{x_i\}_{i\in I}$} the
smallest suboperad of $S$ containing the family $\{x_i\}_{i\in I}$. We write that {\em $x$
is generated by $\{x_i\}_{i\in I}$} if $x$ is in the suboperad generated by $\{x_i\}_{i\in
I}$.

These definitions given, it is natural to search for a smallest family of generators of
$\K \G$. The search of such a family is computationally hard. With the help of the computer,
we obtain that the generators of $\K \G$ of arity no more than $5$ are
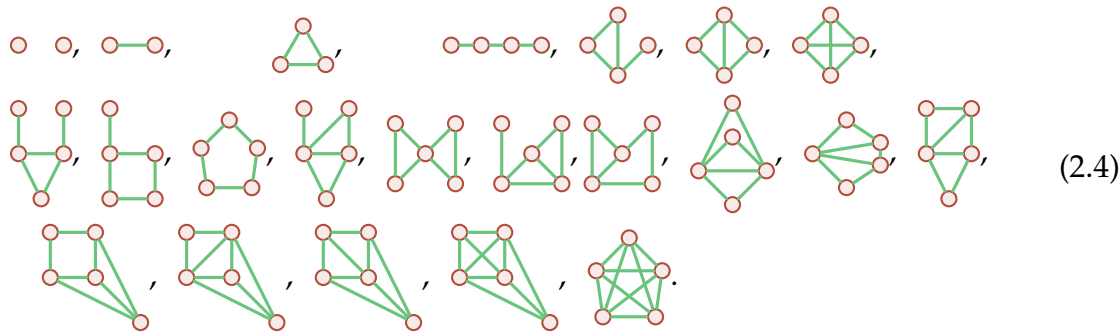
\begin{equation}\begin{split} \label{equ:generators_G}
    &
    \begin{tikzpicture}[Centering,scale=.6]
        \node[UnlabeledNodeGraph](a)at(0,0){};
        \node[UnlabeledNodeGraph](b)at(1,0){};
    \end{tikzpicture},
    \enspace
    \begin{tikzpicture}[Centering,scale=.6]
        \node[UnlabeledNodeGraph](a)at(0,0){};
        \node[UnlabeledNodeGraph](b)at(1,0){};
        \draw[EdgeGraph](a)--(b);
    \end{tikzpicture},
    \qquad \quad
    \begin{tikzpicture}[Centering,scale=.3]
        \node[UnlabeledNodeGraph](a)at(-1,-1){};
        \node[UnlabeledNodeGraph](b)at(1,-1){};
        \node[UnlabeledNodeGraph](c)at(0,.707){};
        \draw[EdgeGraph](a)--(b);
        \draw[EdgeGraph](a)--(c);
        \draw[EdgeGraph](b)--(c);
    \end{tikzpicture},
    \qquad \quad
    \begin{tikzpicture}[Centering,scale=.4]
        \node[UnlabeledNodeGraph](a)at(0,0){};
        \node[UnlabeledNodeGraph](b)at(1,0){};
        \node[UnlabeledNodeGraph](c)at(2,0){};
        \node[UnlabeledNodeGraph](d)at(3,0){};
        \draw[EdgeGraph](a)--(b);
        \draw[EdgeGraph](b)--(c);
        \draw[EdgeGraph](c)--(d);
    \end{tikzpicture},
    \enspace
    \begin{tikzpicture}[Centering,scale=.4]
        \node[UnlabeledNodeGraph](a)at(-1,-1){};
        \node[UnlabeledNodeGraph](b)at(-1,1){};
        \node[UnlabeledNodeGraph](c)at(0,0){};
        \node[UnlabeledNodeGraph](d)at(-2,0){};
        \draw[EdgeGraph](a)--(b);
        \draw[EdgeGraph](a)--(c);
        \draw[EdgeGraph](a)--(d);
        \draw[EdgeGraph](b)--(d);
    \end{tikzpicture},
    \enspace
    \begin{tikzpicture}[Centering,scale=.4]
        \node[UnlabeledNodeGraph](a)at(-1,-1){};
        \node[UnlabeledNodeGraph](b)at(-1,1){};
        \node[UnlabeledNodeGraph](c)at(0,0){};
        \node[UnlabeledNodeGraph](d)at(-2,0){};
        \draw[EdgeGraph](a)--(b);
        \draw[EdgeGraph](a)--(c);
        \draw[EdgeGraph](a)--(d);
        \draw[EdgeGraph](b)--(c);
        \draw[EdgeGraph](b)--(d);
    \end{tikzpicture},
    \enspace
    \begin{tikzpicture}[Centering,scale=.4]
        \node[UnlabeledNodeGraph](a)at(-1,-1){};
        \node[UnlabeledNodeGraph](b)at(-1,1){};
        \node[UnlabeledNodeGraph](c)at(0,0){};
        \node[UnlabeledNodeGraph](d)at(-2,0){};
        \draw[EdgeGraph](a)--(b);
        \draw[EdgeGraph](a)--(c);
        \draw[EdgeGraph](a)--(d);
        \draw[EdgeGraph](b)--(c);
        \draw[EdgeGraph](b)--(d);
        \draw[EdgeGraph](c)--(d);
    \end{tikzpicture},
    \\
    &
    \begin{tikzpicture}[Centering,scale=.3]
        \node[UnlabeledNodeGraph](a)at(-2,0){};
        \node[UnlabeledNodeGraph](b)at(0,0){};
        \node[UnlabeledNodeGraph](c)at(-2,2){};
        \node[UnlabeledNodeGraph](d)at(0,2){};
        \node[UnlabeledNodeGraph](e)at(-1,-2){};
        \draw[EdgeGraph](a)--(b);
        \draw[EdgeGraph](a)--(c);
        \draw[EdgeGraph](a)--(e);
        \draw[EdgeGraph](b)--(d);
        \draw[EdgeGraph](b)--(e);
    \end{tikzpicture},
    \enspace
    \begin{tikzpicture}[Centering,scale=.3]
        \node[UnlabeledNodeGraph](a)at(-2,0){};
        \node[UnlabeledNodeGraph](b)at(0,-2){};
        \node[UnlabeledNodeGraph](c)at(-2,2){};
        \node[UnlabeledNodeGraph](d)at(0,0){};
        \node[UnlabeledNodeGraph](e)at(-2,-2){};
        \draw[EdgeGraph](a)--(c);
        \draw[EdgeGraph](a)--(d);
        \draw[EdgeGraph](a)--(e);
        \draw[EdgeGraph](b)--(d);
        \draw[EdgeGraph](b)--(e);
    \end{tikzpicture},
    \enspace
    \begin{tikzpicture}[Centering,scale=.3]
        \node[UnlabeledNodeGraph](a)at(-1,-1){};
        \node[UnlabeledNodeGraph](b)at(-1.25,.75){};
        \node[UnlabeledNodeGraph](c)at(0,2){};
        \node[UnlabeledNodeGraph](d)at(1.25,.75){};
        \node[UnlabeledNodeGraph](e)at(1,-1){};
        \draw[EdgeGraph](a)--(b);
        \draw[EdgeGraph](b)--(c);
        \draw[EdgeGraph](c)--(d);
        \draw[EdgeGraph](d)--(e);
        \draw[EdgeGraph](e)--(a);
    \end{tikzpicture},
    \enspace
    \begin{tikzpicture}[Centering,scale=.3]
        \node[UnlabeledNodeGraph](a)at(-2,0){};
        \node[UnlabeledNodeGraph](b)at(0,0){};
        \node[UnlabeledNodeGraph](c)at(-2,2){};
        \node[UnlabeledNodeGraph](d)at(0,2){};
        \node[UnlabeledNodeGraph](e)at(-1,-2){};
        \draw[EdgeGraph](a)--(b);
        \draw[EdgeGraph](a)--(c);
        \draw[EdgeGraph](a)--(d);
        \draw[EdgeGraph](a)--(e);
        \draw[EdgeGraph](b)--(d);
        \draw[EdgeGraph](b)--(e);
    \end{tikzpicture},
    \enspace
    \begin{tikzpicture}[Centering,scale=.4]
        \node[UnlabeledNodeGraph](a)at(0,0){};
        \node[UnlabeledNodeGraph](b)at(1,-1){};
        \node[UnlabeledNodeGraph](c)at(-1,1){};
        \node[UnlabeledNodeGraph](d)at(1,1){};
        \node[UnlabeledNodeGraph](e)at(-1,-1){};
        \draw[EdgeGraph](a)--(b);
        \draw[EdgeGraph](a)--(c);
        \draw[EdgeGraph](a)--(d);
        \draw[EdgeGraph](a)--(e);
        \draw[EdgeGraph](b)--(d);
        \draw[EdgeGraph](c)--(e);
    \end{tikzpicture},
    \enspace
    \begin{tikzpicture}[Centering,scale=.4]
        \node[UnlabeledNodeGraph](a)at(0,0){};
        \node[UnlabeledNodeGraph](b)at(1,-1){};
        \node[UnlabeledNodeGraph](c)at(-1,1){};
        \node[UnlabeledNodeGraph](d)at(1,1){};
        \node[UnlabeledNodeGraph](e)at(-1,-1){};
        \draw[EdgeGraph](a)--(b);
        \draw[EdgeGraph](a)--(d);
        \draw[EdgeGraph](a)--(e);
        \draw[EdgeGraph](b)--(d);
        \draw[EdgeGraph](b)--(e);
        \draw[EdgeGraph](c)--(e);
    \end{tikzpicture},
    \begin{tikzpicture}[Centering,scale=.4]
        \node[UnlabeledNodeGraph](a)at(0,0){};
        \node[UnlabeledNodeGraph](b)at(1,-1){};
        \node[UnlabeledNodeGraph](c)at(-1,1){};
        \node[UnlabeledNodeGraph](d)at(1,1){};
        \node[UnlabeledNodeGraph](e)at(-1,-1){};
        \draw[EdgeGraph](a)--(c);
        \draw[EdgeGraph](a)--(d);
        \draw[EdgeGraph](a)--(e);
        \draw[EdgeGraph](b)--(d);
        \draw[EdgeGraph](b)--(e);
        \draw[EdgeGraph](c)--(e);
    \end{tikzpicture},
    \enspace
    \begin{tikzpicture}[Centering,scale=.3]
        \node[UnlabeledNodeGraph](a)at(0,-1){};
        \node[UnlabeledNodeGraph](b)at(3,-1){};
        \node[UnlabeledNodeGraph](c)at(1.5,2){};
        \node[UnlabeledNodeGraph](d)at(1.5,.5){};
        \node[UnlabeledNodeGraph](e)at(1.5,-2.5){};
        \draw[EdgeGraph](a)--(b);
        \draw[EdgeGraph](a)--(c);
        \draw[EdgeGraph](a)--(d);
        \draw[EdgeGraph](a)--(e);
        \draw[EdgeGraph](b)--(c);
        \draw[EdgeGraph](b)--(d);
        \draw[EdgeGraph](b)--(e);
    \end{tikzpicture},
    \enspace
    \begin{tikzpicture}[Centering,scale=.3]
        \node[UnlabeledNodeGraph](a)at(0,-1.5){};
        \node[UnlabeledNodeGraph](b)at(3,-1){};
        \node[UnlabeledNodeGraph](c)at(3,-2){};
        \node[UnlabeledNodeGraph](d)at(1.5,0){};
        \node[UnlabeledNodeGraph](e)at(1.5,-3){};
        \draw[EdgeGraph](a)--(b);
        \draw[EdgeGraph](a)--(c);
        \draw[EdgeGraph](a)--(d);
        \draw[EdgeGraph](a)--(e);
        \draw[EdgeGraph](b)--(c);
        \draw[EdgeGraph](b)--(d);
        \draw[EdgeGraph](c)--(e);
    \end{tikzpicture},
    \enspace
    \begin{tikzpicture}[Centering,scale=.6]
        \node[UnlabeledNodeGraph](a)at(0,0){};
        \node[UnlabeledNodeGraph](b)at(1,1){};
        \node[UnlabeledNodeGraph](c)at(.5,-1){};
        \node[UnlabeledNodeGraph](d)at(0,1){};
        \node[UnlabeledNodeGraph](e)at(1,0){};
        \draw[EdgeGraph](a)--(b);
        \draw[EdgeGraph](a)--(c);
        \draw[EdgeGraph](a)--(d);
        \draw[EdgeGraph](a)--(e);
        \draw[EdgeGraph](b)--(d);
        \draw[EdgeGraph](b)--(e);
        \draw[EdgeGraph](c)--(e);
    \end{tikzpicture},
    \\
    &
    \quad
    \begin{tikzpicture}[Centering,scale=.6]
        \node[UnlabeledNodeGraph](a)at(0,0){};
        \node[UnlabeledNodeGraph](b)at(1,1){};
        \node[UnlabeledNodeGraph](c)at(2,-1){};
        \node[UnlabeledNodeGraph](d)at(0,1){};
        \node[UnlabeledNodeGraph](e)at(1,0){};
        \draw[EdgeGraph](a)--(c);
        \draw[EdgeGraph](a)--(d);
        \draw[EdgeGraph](a)--(e);
        \draw[EdgeGraph](b)--(c);
        \draw[EdgeGraph](b)--(d);
        \draw[EdgeGraph](b)--(e);
        \draw[EdgeGraph](c)--(e);
    \end{tikzpicture},
    \enspace
    \begin{tikzpicture}[Centering,scale=.6]
        \node[UnlabeledNodeGraph](a)at(0,0){};
        \node[UnlabeledNodeGraph](b)at(1,1){};
        \node[UnlabeledNodeGraph](c)at(2,-1){};
        \node[UnlabeledNodeGraph](d)at(0,1){};
        \node[UnlabeledNodeGraph](e)at(1,0){};
        \draw[EdgeGraph](a)--(b);
        \draw[EdgeGraph](a)--(c);
        \draw[EdgeGraph](a)--(d);
        \draw[EdgeGraph](a)--(e);
        \draw[EdgeGraph](b)--(c);
        \draw[EdgeGraph](b)--(d);
        \draw[EdgeGraph](b)--(e);
        \draw[EdgeGraph](c)--(e);
    \end{tikzpicture},
    \enspace
    \begin{tikzpicture}[Centering,scale=.6,rotate=90]
        \node[UnlabeledNodeGraph](a)at(0,0){};
        \node[UnlabeledNodeGraph](b)at(1,1){};
        \node[UnlabeledNodeGraph](c)at(-1,-1){};
        \node[UnlabeledNodeGraph](d)at(0,1){};
        \node[UnlabeledNodeGraph](e)at(1,0){};
        \draw[EdgeGraph](a)--(b);
        \draw[EdgeGraph](a)--(c);
        \draw[EdgeGraph](a)--(d);
        \draw[EdgeGraph](a)--(e);
        \draw[EdgeGraph](b)--(d);
        \draw[EdgeGraph](b)--(e);
        \draw[EdgeGraph](c)--(d);
        \draw[EdgeGraph](c)--(e);
    \end{tikzpicture},
    \enspace
    \begin{tikzpicture}[Centering,scale=.6,rotate=90]
        \node[UnlabeledNodeGraph](a)at(0,0){};
        \node[UnlabeledNodeGraph](b)at(1,1){};
        \node[UnlabeledNodeGraph](c)at(-1,-1){};
        \node[UnlabeledNodeGraph](d)at(0,1){};
        \node[UnlabeledNodeGraph](e)at(1,0){};
        \draw[EdgeGraph](a)--(b);
        \draw[EdgeGraph](a)--(c);
        \draw[EdgeGraph](a)--(d);
        \draw[EdgeGraph](a)--(e);
        \draw[EdgeGraph](b)--(d);
        \draw[EdgeGraph](b)--(e);
        \draw[EdgeGraph](c)--(d);
        \draw[EdgeGraph](c)--(e);
        \draw[EdgeGraph](d)--(e);
    \end{tikzpicture},
    \enspace
    \begin{tikzpicture}[Centering,scale=.35]
        \node[UnlabeledNodeGraph](a)at(-1,-1){};
        \node[UnlabeledNodeGraph](b)at(-1.25,.75){};
        \node[UnlabeledNodeGraph](c)at(0,2){};
        \node[UnlabeledNodeGraph](d)at(1.25,.75){};
        \node[UnlabeledNodeGraph](e)at(1,-1){};
        \draw[EdgeGraph](a)--(b);
        \draw[EdgeGraph](a)--(c);
        \draw[EdgeGraph](a)--(d);
        \draw[EdgeGraph](a)--(e);
        \draw[EdgeGraph](b)--(c);
        \draw[EdgeGraph](b)--(d);
        \draw[EdgeGraph](b)--(e);
        \draw[EdgeGraph](c)--(d);
        \draw[EdgeGraph](c)--(e);
        \draw[EdgeGraph](d)--(e);
    \end{tikzpicture}.
\end{split}\end{equation}
Due to the symmetric group action on $\K \G$, only the knowledge of the shapes of the graphs
is significant.  While~\eqref{equ:generators_G} does not provide to us any particular
insight on a possible characterisation of the generators, it does suggest that any graph
with enough edges must be a generator. This is confirmed by the following lemma.

\begin{lemma} \label{lem:infinite_number_generators_G}
Let $\{V_i\}_{i\in I}$ be a family of non empty finite sets, $\{g_i\}_{i \in I}$ be a family
of graphs such that $g_i \in\G[V_i]$, and let $g$ be a graph in $\G[V]$ with at least
$\binom{n-1}{2} +1$ edges, where $n = |V|$. Then $g$ is generated by $\{g_i\}_{i\in I}$ if
and only if $g=g_i$ for some $i\in I$.
\end{lemma}
\begin{proof}[Sketch of proof.]
Remark that the number of edges of the graphs in the support of $g_1\circ_\ast g_2$ is the
sum of the number of edges in $g_1$ with the number of edges in $g_2$. Hence graphs with too
many edges cannot appear in the support of a partial composition.
\end{proof}

\begin{proposition} \label{gc}
    The operad $\K \G$ is not free and has an infinite number of generators.
\end{proposition}
\begin{proof}
The fact that $\K \G$ has an infinite number of generators is a direct consequence of
Lemma~\ref{lem:infinite_number_generators_G}. Moreover, the relation
\begin{equation}\begin{split} \label{nf}
    \begin{tikzpicture}[Centering,scale=.7]
        \node[NodeGraph](a)at(0,0){$a$};
        \node[NodeGraph](s)at(1,0){$\ast$};
        \draw[EdgeGraph](a)--(s);
    \end{tikzpicture}
    &
    \circ_\ast
    \begin{tikzpicture}[Centering,scale=.7]
        \node[NodeGraph](b)at(0,0){$b$};
        \node[NodeGraph](c)at(1,0){$c$};
        \draw[EdgeGraph](b)--(c);
    \end{tikzpicture}
    \enspace + \enspace
    \begin{tikzpicture}[Centering,scale=.7]
        \node[NodeGraph](c)at(0,0){$c$};
        \node[NodeGraph](s)at(1,0){$\ast$};
        \draw[EdgeGraph](c)--(s);
    \end{tikzpicture}
    \circ_\ast
    \begin{tikzpicture}[Centering,scale=.7]
        \node[NodeGraph](b)at(0,0){$b$};
        \node[NodeGraph](a)at(1,0){$a$};
        \draw[EdgeGraph](b)--(a);
    \end{tikzpicture}
    \enspace - \enspace
    \begin{tikzpicture}[Centering,scale=.7]
        \node[NodeGraph](b)at(0,0){$b$};
        \node[NodeGraph](s)at(1,0){$\ast$};
        \draw[EdgeGraph](b)--(s);
    \end{tikzpicture}
    \circ_\ast
    \begin{tikzpicture}[Centering,scale=.7]
        \node[NodeGraph](a)at(0,0){$a$};
        \node[NodeGraph](c)at(1,0){$c$};
        \draw[EdgeGraph](a)--(c);
    \end{tikzpicture}
    \enspace
    - 2 \,
    \begin{tikzpicture}[Centering,scale=.7]
        \node[NodeGraph](a)at(0,0){$a$};
        \node[NodeGraph](b)at(1,0){$b$};
        \node[NodeGraph](c)at(2,0){$c$};
        \draw[EdgeGraph](a)--(b);
        \draw[EdgeGraph](b)--(c);
    \end{tikzpicture}
    \\[.5em]
    & =
    \begin{tikzpicture}[Centering,scale=.7]
        \node[NodeGraph](a)at(0,0){$a$};
        \node[NodeGraph](b)at(1,0){$b$};
        \node[NodeGraph](c)at(2,0){$c$};
        \draw[EdgeGraph](a)--(b);
        \draw[EdgeGraph](b)--(c);
    \end{tikzpicture}
    \enspace + \enspace
    \begin{tikzpicture}[Centering,scale=.7]
        \node[NodeGraph](b)at(0,0){$b$};
        \node[NodeGraph](c)at(1,0){$c$};
        \node[NodeGraph](a)at(2,0){$a$};
        \draw[EdgeGraph](b)--(c);
        \draw[EdgeGraph](c)--(a);
    \end{tikzpicture}
    \enspace + \enspace
    \begin{tikzpicture}[Centering,scale=.7]
        \node[NodeGraph](c)at(0,0){$c$};
        \node[NodeGraph](b)at(1,0){$b$};
        \node[NodeGraph](a)at(2,0){$a$};
        \draw[EdgeGraph](c)--(b);
        \draw[EdgeGraph](b)--(a);
    \end{tikzpicture}
    \enspace + \enspace
    \begin{tikzpicture}[Centering,scale=.7]
        \node[NodeGraph](b)at(0,0){$b$};
        \node[NodeGraph](a)at(1,0){$a$};
        \node[NodeGraph](c)at(2,0){$c$};
        \draw[EdgeGraph](b)--(a);
        \draw[EdgeGraph](a)--(c);
    \end{tikzpicture}
    \\[.5em]
    & \qquad - \enspace
    \begin{tikzpicture}[Centering,scale=.7]
        \node[NodeGraph](b)at(0,0){$b$};
        \node[NodeGraph](a)at(1,0){$a$};
        \node[NodeGraph](c)at(2,0){$c$};
        \draw[EdgeGraph](b)--(a);
        \draw[EdgeGraph](a)--(c);
    \end{tikzpicture}
    \enspace - \enspace
    \begin{tikzpicture}[Centering,scale=.7]
        \node[NodeGraph](a)at(0,0){$a$};
        \node[NodeGraph](c)at(1,0){$c$};
        \node[NodeGraph](b)at(2,0){$b$};
        \draw[EdgeGraph](a)--(c);
        \draw[EdgeGraph](c)--(b);
    \end{tikzpicture}
    - 2 \,
    \begin{tikzpicture}[Centering,scale=.7]
        \node[NodeGraph](a)at(0,0){$a$};
        \node[NodeGraph](b)at(1,0){$b$};
        \node[NodeGraph](c)at(2,0){$c$};
        \draw[EdgeGraph](a)--(b);
        \draw[EdgeGraph](b)--(c);
    \end{tikzpicture}
    \\[.5em]
    & = 0
\end{split}\end{equation}
shows that $\K \G$ is not free.
\end{proof}

As a consequence of Proposition~\ref{gc}, it seems particularly hard to further investigate
the structure of $\K \G$. Let us restrict further to its suboperad $\K \T$ of trees. The
generators of $\K \T$ until arity $6$ are
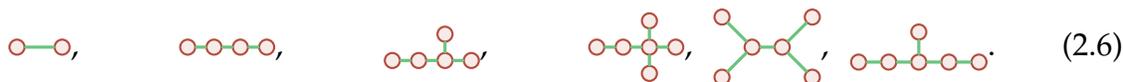
\begin{equation}\begin{split}
    \begin{tikzpicture}[Centering,scale=.6]
        \node[UnlabeledNodeGraph](a)at(0,0){};
        \node[UnlabeledNodeGraph](b)at(1,0){};
        \draw[EdgeGraph](a)--(b);
    \end{tikzpicture},
    \qquad \quad
    \begin{tikzpicture}[Centering,scale=.35]
        \node[UnlabeledNodeGraph](a)at(0,0){};
        \node[UnlabeledNodeGraph](b)at(1,0){};
        \node[UnlabeledNodeGraph](c)at(2,0){};
        \node[UnlabeledNodeGraph](d)at(3,0){};
        \draw[EdgeGraph](a)--(b);
        \draw[EdgeGraph](b)--(c);
        \draw[EdgeGraph](c)--(d);
    \end{tikzpicture},
    \qquad \quad
    \begin{tikzpicture}[Centering,scale=.35]
        \node[UnlabeledNodeGraph](a)at(0,0){};
        \node[UnlabeledNodeGraph](b)at(-1,0){};
        \node[UnlabeledNodeGraph](c)at(0,1){};
        \node[UnlabeledNodeGraph](d)at(1,0){};
        \node[UnlabeledNodeGraph](e)at(-2,0){};
        \draw[EdgeGraph](a)--(b);
        \draw[EdgeGraph](a)--(c);
        \draw[EdgeGraph](a)--(d);
        \draw[EdgeGraph](b)--(e);
    \end{tikzpicture},
    \qquad \quad
    \begin{tikzpicture}[Centering,scale=.35]
        \node[UnlabeledNodeGraph](a)at(0,0){};
        \node[UnlabeledNodeGraph](b)at(-1,0){};
        \node[UnlabeledNodeGraph](c)at(0,1){};
        \node[UnlabeledNodeGraph](d)at(1,0){};
        \node[UnlabeledNodeGraph](e)at(0,-1){};
        \node[UnlabeledNodeGraph](f)at(-2,0){};
        \draw[EdgeGraph](a)--(b);
        \draw[EdgeGraph](a)--(c);
        \draw[EdgeGraph](a)--(d);
        \draw[EdgeGraph](a)--(e);
        \draw[EdgeGraph](b)--(f);
    \end{tikzpicture},
    \enspace
    \begin{tikzpicture}[Centering,scale=.4]
        \node[UnlabeledNodeGraph](a)at(0,0){};
        \node[UnlabeledNodeGraph](b)at(-1,0){};
        \node[UnlabeledNodeGraph](c)at(1,-1){};
        \node[UnlabeledNodeGraph](d)at(1,1){};
        \node[UnlabeledNodeGraph](e)at(-2,-1){};
        \node[UnlabeledNodeGraph](f)at(-2,1){};
        \draw[EdgeGraph](a)--(b);
        \draw[EdgeGraph](a)--(c);
        \draw[EdgeGraph](a)--(d);
        \draw[EdgeGraph](b)--(e);
        \draw[EdgeGraph](b)--(f);
    \end{tikzpicture},
    \enspace
    \begin{tikzpicture}[Centering,scale=.4]
        \node[UnlabeledNodeGraph](a)at(0,0){};
        \node[UnlabeledNodeGraph](b)at(1,0){};
        \node[UnlabeledNodeGraph](c)at(2,0){};
        \node[UnlabeledNodeGraph](d)at(3,0){};
        \node[UnlabeledNodeGraph](e)at(4,0){};
        \node[UnlabeledNodeGraph](f)at(2,1){};
        \draw[EdgeGraph](a)--(b);
        \draw[EdgeGraph](b)--(c);
        \draw[EdgeGraph](c)--(d);
        \draw[EdgeGraph](d)--(e);
        \draw[EdgeGraph](c)--(f);
    \end{tikzpicture}.
\end{split}\end{equation}
%In the same way as before, one can show that $\K \T$ has an infinite number of generators.
This operad $\K \T$ has a non trivial link with the pre-Lie operad $\PLie$~\cite{CL01}. To
show this we first need to introduce a new operad on oriented multigraphs.

An \textit{oriented multigraph} on $V$ is a graph where each edge end is either unlabelled
or labelled with an arrow head. We denote by $\MG_{or}$ the set species of oriented graphs,
by $\G_{or}$ the set species of oriented graphs, and by $\MG_{orc}$ and $\G_{orc}$ their
connected counterparts.

Let $V_1$ and $V_2$ be two disjoint sets such that $\ast \in V_1$. For any rooted oriented
multigraphs $(g_1, v_1) \in \MG_{or}^\bullet[V_1]$ and $(g_2, v_2) \in
\MG_{or}[V_2]^\bullet$, the \textit{rooted insertion} of $(g_2, v_2)$ into $(g_1, v_1)$ is
the sum of all the rooted multigraphs of $\MG_{or}^\bullet[V_1 \setminus \{\ast\} \sqcup
V_2]$ obtained by the following process:
\begin{enumerate}
    \item Do the disjoint union of $g_1$ and $g_2$;

    \item Remove the vertex $\ast$. We then have some edges with a loose end;

    \item Connect each non labelled loose end to $v_2$;

    \item Connect each labelled loose end to any vertex in $V_2$;

    \item The new root is $v_1$ if $v_1 \ne \ast$ and is $v_2$ otherwise.
\end{enumerate}
For instance, by depicting by squares the roots of the graphs,
\begin{equation}
 \label{canoor}
    \begin{tikzpicture}[Centering,scale=.7]
        \node[NodeGraph](s)at(0,0){$\ast$};
        \node[RootGraph](a)at(1,1){$a$};
        \node[NodeGraph](b)at(1,-1){$b$};
        \draw[ArcGraph](a)--(s);
        \draw[ArcGraph](s)--(b);
    \end{tikzpicture}
    \enspace \circ_\ast \enspace
    \begin{tikzpicture}[Centering,scale=1]
        \node[RootGraph](c)at(0,0){$c$};
        \node[NodeGraph](d)at(1,0){$a$};
        \draw[ArcGraph](c)--(d);
    \end{tikzpicture}
    \enspace = \enspace
    \begin{tikzpicture}[Centering,scale=.7]
        \node[RootGraph](a)at(1,1){$a$};
        \node[NodeGraph](b)at(1,-1){$b$};
        \node[NodeGraph](c)at(0,0){$c$};
        \node[NodeGraph](d)at(2,0){$d$};
        \draw[ArcGraph](a)--(c);
        \draw[ArcGraph](c)--(b);
        \draw[ArcGraph](c)--(d);
    \end{tikzpicture}
    \enspace + \enspace
    \begin{tikzpicture}[Centering,scale=.7]
        \node[RootGraph](a)at(1,1){$a$};
        \node[NodeGraph](b)at(1,-1){$b$};
        \node[NodeGraph](c)at(0,0){$c$};
        \node[NodeGraph](d)at(2,0){$d$};
        \draw[ArcGraph](a)--(d);
        \draw[ArcGraph](c)--(d);
        \draw[ArcGraph](c)--(b);
    \end{tikzpicture}
\end{equation}

\begin{theorem}
    The species $\K \MG_{orc}^\bullet$, endowed with the rooted insertion as partial
    composition, is an operad.
\end{theorem}

This makes $\K \G_{orc}^{\bullet}$ a suboperad of $\K \MG_{orc}^{\bullet}$.

In a rooted tree, each edge has a parent end and a child end. Given a rooted tree $t$ with
root $r$, denote by $t_r$ the oriented tree where each parent end of $t$ is labelled and
each child end is non labelled. Then, the monomorphism $\T^{\bullet}\hookrightarrow
\G_{orc}^{\bullet}$ which sends each ordered pair $(t,r)$, where $t$ is a tree and $r$ is
its root, on $(t_r,r)$ induces an operad structure on the species of rooted trees which is
exactly the operad $\PLie$.

\begin{proposition} \label{prelie}
The monomorphism of species $\psi : \K \T \to \K \T^{\bullet}$ defined, for any tree $t \in
\T[V]$ by
\begin{equation}
    \psi(t) =  \sum_{r \in V} (t, r),
\end{equation}
is a monomorphism of operads from $\K \T$ to $\PLie$.
\end{proposition}

A natural question to ask is how to extend this morphism to $\K \G_c$ and $\K \MG_c$. Let us
introduce some notations in order to answer this question. For $g\in \MG_c[V]$, $r\in V$,
and $t\in \T[V]$ a spanning tree of $g$, let $\overrightarrow{g}^{(t,r)}\in \MG_{orc}$ be
the oriented multigraph obtained by labelling the edges of $g$ in $t$ in the same way as the
edges of $t_r$, and by labelling both ends of the edges in $g$ not in $t$. More formally, we
have: $\overrightarrow{g}^{(t,r)}=t_{r}\oplus\iota_{\G}(g\setminus t)$, where $\iota: \K
\MG\rightarrow \K \MG_{or}$ sends a multigraph to the oriented multigraph obtained by
labelling all the edges ends.

Define $\K \Operad_2\subset \K\Operad_1\subset\K \ST$ three subspecies of $\K
\MG_{orc}^{\bullet}$ by
\begin{equation}
    \ST[V]=\left\{
        (\overrightarrow{g}^{(t,r)},r) : g\in \MG_c[V], r\in V \text{ and $t$ a
    spanning tree of $g$}\right\},
\end{equation}
\begin{equation}
    \Operad_1[V] = \left\{\sum_{r\in V} (\overrightarrow{g}^{(t(r),r)},r) : g\in
    \MG_c[V]\text{ and for each $r$, $t(r)$ a spanning tree of $g$}\right\},
\end{equation}
\begin{multline}
    \Operad_2[V]=
    \left\{(\overrightarrow{g}^{(t_1,r)},r)-(\overrightarrow{g}^{(t_2,r)},r) : g\in
    \MG_c[V],r\in V,
    \right. \\ \left.
    \text{ and $t_1$ and $t_2$ two spanning trees of $g$}\right\}.
\end{multline}

\begin{lemma} \label{lemmfond}
The following properties hold
\begin{itemize}
\item $\K \ST$ is a suboperad of $\K \MG_{orc}^{\bullet}$ isomorphic to $\K\MG\times\PLie$,
\item $\K \Operad_1$ is a suboperad of $\K \ST$, 
\item $\K\Operad_2$ is an ideal of $\K\Operad_1$.
\end{itemize}
\end{lemma}

We can see $\PLie$ as a suboperad of $\ST$ by the monomorphism $(t,r)\mapsto (t_r,r)$. The
image of the operad morphism $\psi$ of Proposition~\ref{prelie} is then $\K\Operad_1\cap
\PLie$ and we have that $\K\Operad_2\cap \PLie = \{0\}$ and hence $\K\Operad_1\cap
\PLie/\K\Operad_2\cap \PLie = \K\Operad_1\cap \PLie$.

\begin{proposition}
The operad isomorphism $\psi: \K \T \to \PLie$
extends into an operad isomorphism $\psi: \K \MG_c \to \K\Operad_1/\K\Operad_2$ satisfying,
for any $g \in \MG_c[V]$,
\begin{equation}
    \psi(g) = \sum_{r\in V}\overrightarrow{g}^{(t(r),r)},
\end{equation}
where for each $r\in V$, $t(r)$ is a spanning tree of $g$. Furthermore, this isomorphism
restricts itself to an isomorphism $\K\G_C \to \K\Operad_1\cap\K\G_c/\K\Operad_2\cap\K\G_c$.
\end{proposition}

The last results are summarized in the following commutative diagram of operad morphisms.
\begin{equation}
\begin{tikzcd}
    \K \T \arrow[r, "\sim"] \arrow[d, hook]
    & \PLie\cap\K\Operad_1/\K\Operad_2 \arrow[r, equal] \arrow[d, hook]
    & \PLie\cap\K\Operad_1 \arrow[d,hook] \arrow[r, hook]
    & \PLie \arrow[d,hook]\\
    \K \G_c \arrow[r, "\sim"] \arrow[d, hook]
    & \K\Operad_1\cap\K\G_{orc}^{\bullet}/\K\Operad_2\cap\K\G_{orc}^{\bullet} \arrow[d, hook]
    & \K\G_{orc}^{\bullet}\cap\K\Operad_1 \arrow[l, two heads] \arrow[r, hook] \arrow[d, hook]
    & \K\G_{orc}^{\bullet}\cap\K \ST \arrow[d, hook] \\
    \K\MG_c \arrow[r, "\sim"]
    & \K\Operad_1/\Operad_2
    & \K\Operad_1 \arrow[l, two heads] \arrow[r, hook]
    & \K\MG\times\PLie
\end{tikzcd}
\end{equation}

%%%%%%%%%%%%%%%%%%%%%%%%%%%%%%%%%%%%%%%%%%%%%%%%%%%%%%%%%%%%%%%%%%%%%%%%%%%%%%%%%%%%%%%%%%%%
%%%%%%%%%%%%%%%%%%%%%%%%%%%%%%%%%%%%%%%%%%%%%%%%%%%%%%%%%%%%%%%%%%%%%%%%%%%%%%%%%%%%%%%%%%%%
\section{Finitely generated suboperads} \label{sec:suboperads}
Let us now focus on finitely generated suboperads of $\K \G$.  First remark that the
suboperad of $\K \G$ generated by
\begin{math}
    \left\{
    \begin{tikzpicture}[Centering,scale=.6]
        \node[NodeGraph](a)at(0,0){$a$};
        \node[NodeGraph](b)at(1,0){$b$};
    \end{tikzpicture}
    \right\}
\end{math}
is isomorphic to the commutative operad $\K \Com$. Indeed,
\begin{equation}
    \begin{tikzpicture}[Centering,scale=.7]
        \node[NodeGraph](a)at(0,0){$a$};
        \node[NodeGraph](s)at(1,0){$\ast$};
    \end{tikzpicture}
    \enspace \circ_\ast \enspace
    \begin{tikzpicture}[Centering,scale=.7]
        \node[NodeGraph](b)at(0,0){$b$};
        \node[NodeGraph](c)at(1,0){$c$};
    \end{tikzpicture}
    \enspace = \enspace
    \begin{tikzpicture}[Centering,scale=.7]
        \node[NodeGraph](a)at(0,0){$a$};
        \node[NodeGraph](b)at(1,0){$b$};
        \node[NodeGraph](c)at(2,0){$c$};
    \end{tikzpicture}
    \enspace = \enspace
    \begin{tikzpicture}[Centering,scale=.7]
        \node[NodeGraph](s)at(0,0){$\ast$};
        \node[NodeGraph](c)at(1,0){$c$};
    \end{tikzpicture}
    \enspace \circ_\ast \enspace
    \begin{tikzpicture}[Centering,scale=.7]
        \node[NodeGraph](a)at(0,0){$a$};
        \node[NodeGraph](b)at(1,0){$b$};
    \end{tikzpicture}
 \end{equation}

Recall that the set operad $\ComMag$~\cite{BL11} is the free set operad generated by one
binary and symmetric element. More formally, $\ComMag[V]$ is the set of all nonplanar binary
trees with set of leaves equal to $V$. Let $s$ be the connected set species defined by
$|s[V]| = 1$ if $|V|=2$, $|s[V]|=0$ otherwise. The action of transposition on the sole
element of $s[\{a,b\}]$ is trivial. Then $\K \ComMag = \FreeOp_{\K s}$.
\begin{proposition} \label{commag}
    The suboperad of $\K \G$ generated by
    \begin{math}
        \left\{
        \begin{tikzpicture}[Centering,scale=.6]
            \node[NodeGraph](a)at(0,0){$a$};
            \node[NodeGraph](b)at(1,0){$b$};
            \draw[EdgeGraph](a)--(b);
        \end{tikzpicture}
        \right\}
    \end{math}
    is isomorphic to $\K \ComMag$.
\end{proposition}
\begin{proof}
We know from Proposition~\ref{prelie} that the operad of the statement is isomorphic to the
suboperad of $\PLie$ generated by
\begin{equation}
    \left\{
    \begin{tikzpicture}[Centering,scale=.6]
        \node[RootGraph](a)at(0,0){$a$};
        \node[NodeGraph](b)at(0,-1){$b$};
        \draw[EdgeGraph](a)--(b);
    \end{tikzpicture}
    +
    \begin{tikzpicture}[Centering,scale=.6]
        \node[RootGraph](a)at(0,0){$b$};
        \node[NodeGraph](b)at(0,-1){$a$};
        \draw[EdgeGraph](a)--(b);
    \end{tikzpicture}
    \right\}
\end{equation}
Then~\cite{BL11} gives us that this suboperad is isomorphic to $\K \ComMag$. This concludes
the proof
\end{proof}

Now the fact that we can see both $\K \Com$ and $\K \ComMag$ as suboperads of $\K \G$ gives
us natural way to define the smallest operad containing these two as suboperads. Let $\SP$
be the suboperad of $\K \G$ generated by
\begin{math}
    \left\{
    \begin{tikzpicture}[Centering,scale=.6]
        \node[NodeGraph](a)at(0,0){$a$};
        \node[NodeGraph](b)at(1,0){$b$};
    \end{tikzpicture},
    \begin{tikzpicture}[Centering,scale=.6]
        \node[NodeGraph](a)at(0,0){$a$};
        \node[NodeGraph](b)at(1,0){$b$};
        \draw[EdgeGraph](a)--(b);
    \end{tikzpicture}
    \right\}
\end{math}
This operad has some nice properties.

\begin{proposition}
    The operad $\SP$ is isomorphic to the operad $\Ope(G, R)$ where $G$ is the subspecies of
    $\K\G$ generated by $\{\Points{a}{b}, \Segment{a}{b}\}$ and $R$ is the subspecies of
    $\FreeOp_G$ generated by
     \begin{subequations}
    \begin{equation} \label{equ:rel_1}
        \Points{c}{\ast} \circ^{\xi}_\ast \Points{a}{b}
        \enspace - \enspace
        \Points{a}{\ast} \circ^{\xi}_\ast \Points{b}{c},
    \end{equation}
    \begin{center}
        and
    \end{center}
    \begin{equation} \label{equ:rel_2}
        \Segment{a}{\ast} \circ^{\xi}_\ast \Points{b}{c}
        \enspace - \enspace
        \Points{c}{\ast} \circ^{\xi}_\ast \Segment{a}{b}
        \enspace - \enspace
        \Points{b}{\ast} \circ^{\xi}_\ast \Segment{a}{c}.
    \end{equation}
    \end{subequations}
    Therefore, $\SP$ is a binary and quadratic operad.
\end{proposition}

For the readers familiar with Koszulity (see~\cite{LV12}), remark that $\SP$ is a Koszul
operad.

\begin{proposition}
    The operad $\SP$ admits as Koszul dual the operad $\SP^!$ which is isomorphic to the
    operad $\Ope(G^{\vee}, R)$ where $G$ is the subspecies of $\K\G$ generated by
    $\{\Points{a}{b}^{\vee}, \Segment{a}{b}^{\vee}\}$ and $R$ is the subspecies of
    $\FreeOp_{G^\vee}$ generated by
    \begin{subequations}
    \begin{equation} \label{equ:rel_dual_1}
        \Segment{a}{\ast}^{\vee} \circ^{\xi}_\ast \Segment{b}{c}^{\vee},
    \end{equation}
    \begin{equation} \label{equ:rel_dual_2}
        \Points{a}{\ast}^{\vee} \circ^{\xi}_\ast \Segment{b}{c}^{\vee}
        \enspace + \enspace
        \Segment{c}{\ast}^{\vee} \circ^{\xi}_\ast \Points{a}{b}^{\vee}
        \enspace + \enspace
        \Segment{b}{\ast}^{\vee} \circ^{\xi}_\ast \Points{a}{c}^{\vee},
    \end{equation}
    \begin{equation} \label{equ:rel_dual_3}
        \Points{a}{\ast}^{\vee} \circ^{\xi}_\ast \Points{b}{c}^{\vee}
        +
        \Points{c}{\ast}^{\vee} \circ^{\xi}_\ast \Points{a}{b}^{\vee}
        +
        \Points{b}{\ast}^{\vee} \circ^{\xi}_\ast \Points{c}{a}^{\vee}.
    \end{equation}
    \end{subequations}
\end{proposition}
\begin{proof}[Sketch of proof]
Let us respectively denote by $r_1$, $r_2$, $r'_1$, $r'_2$, and $r'_3$ the
elements~\eqref{equ:rel_1}, \eqref{equ:rel_2}, \eqref{equ:rel_dual_1},
\eqref{equ:rel_dual_2}, and~\eqref{equ:rel_dual_3}.  Denote by $I$ the operad ideal
generated by~$r_1$ and~$r_2$. Then as a vector space, $I[[\{a,b,c\}]]$ is the linear span of
the set
\begin{equation}
    \{r_1, r_1 \cdot (ab), r_2,r_2 \cdot (abc), r_2 \cdot (acb)\},
\end{equation}
where $\cdot$ is the action of the symmetric group, e.g $r_1\cdot (ab) =
\FreeOp_{G}[(ab)](r_1)$.  This space is a sub-space of dimension $5$ of
$\FreeOp_G[\{a,b,c\}$, which is of dimension $12$. Hence, since as a vector space
\begin{equation}
    \FreeOp_{G^\vee}[\{a,b,c\}]
    \cong \FreeOp_{G^*}[\{a,b,c\}]\cong
    \FreeOp_{G}[\{a,b,c\}],
\end{equation}
$I^{\bot}[\{a,b,c\}]$ must be of dimension $7$.

Denote by $J$ the ideal generated by $r_1'$, $r_2'$ and $r_3'$. Then as a vector space
$J[\{a,b,c\}]$ is the linear span of the set
\begin{equation}
    \{ r_1', r_1'\cdot (ab), r_1'\cdot (ac),r_2', r_2'\cdot (abc),r_2'\cdot (acb), r_3'\}.
\end{equation}
This space is of dimension 7. Verifying that for any $f\in J[\{a,b,c\}]$ and $x\in
I[\{a,b,c\}]$ we have $<f,x>=0$ concludes this proof.
% To conclude we need to show that for
% any $f\in J[\{a,b,c\}]$ and $x\in I[\{a,b,c\}]$, $<f,x>=0$. For example, denoting by $p_{a,b} =
% \Points{a}{b}$, $s_{a,b} = \Segment{a}{b}$ and $p_{a,b}^{\vee}$ and
% $s_{a,b}^{\vee}$ their duals in $G[\{a,b\}]^{\vee}$ we have
% \begin{align*} <r_1',r_1> &=
% <s_{a,\ast}^{\vee}\circ_{\ast} s_{b,c}^{\vee}\,,\, p_{\ast,c}\circ_{\ast}p_{a,b} -
% p_{a,\ast}\circ_{\ast} p_{b,c}> \\
% &= <s_{a,\ast}^{\vee}\circ_{\ast} s_{b,c}^{\vee}\,,\, p_{\ast,c}\circ_{\ast}p_{a,b}> - <s_{a,\ast}^{\vee}\circ_{\ast} s_{b,c}^{\vee}\,,\,p_{a,\ast}\circ_{\ast} p_{b,c}> \\
% &= s_{a,\ast}^{\vee}(p_{\ast,c})s_{b,c}^{\vee}(p_{a,b}) -
% s_{a,\ast}^{\vee}(p_{a,\ast})s_{b,c}^{\vee}(p_{b,c}) \\ &= 0.
% \end{align*}
\end{proof}

\begin{proposition}
    The Hilbert series of $\SP^{!}$ is 
    \begin{equation}
        \mathcal{H}_{\SP^!}(x) = \dfrac{(1-\log(1-x))^2-1}{2}.
    \end{equation}
\end{proposition}

The first dimensions $\dim \SP^![[n]]$ for $n\geq 1$ are
\begin{equation}
    1, 2, 5, 17, 74, 394, 2484, 18108, 149904.
\end{equation}
This is sequence~\OEIS{A000774} of~\cite{Slo}. This sequence is in particular linked to some
pattern avoiding signed permutations and mesh patterns.

Before ending this section let us mention the suboperad $\LP$ of $\K \MG$ generated by
\begin{equation}
    \left\{
    \begin{tikzpicture}[Centering,scale=.6]
        \tikzset{every loop/.style={}}
        \node[NodeGraph](a)at(0,0){$a$};
        \draw[EdgeGraph](a)edge[loop](a);
    \end{tikzpicture},
    \begin{tikzpicture}[Centering,scale=.6]
        \node[NodeGraph](a)at(0,0){$a$};
        \node[NodeGraph](b)at(1,0){$b$};
    \end{tikzpicture}
    \right\}.
\end{equation}
This operad presents a clear interest since its two generators can be considered as minimal
elements in the sense that a partial composition with the two isolated vertices adds exactly
one vertex and no edges, while a partial composition with the loop adds exactly one edge and
no vertex.  A natural question to ask at this point concerns the description of the
multigraphs generated by these two minimal elements.

\begin{proposition}
The following properties hold
\begin{itemize}
\item the operad $\SP$ is a suboperad of $\LP$;
\item the operad $\LP$ is a strict suboperad of $\K \MG$. In particular, the multigraph
\begin{equation}
    \begin{tikzpicture}[Centering,scale=.8]
        \node[NodeGraph](a)at(0,0){$a$};
        \node[NodeGraph](b)at(1,0){$b$};
        \node[NodeGraph](c)at(2,0){$c$};
        \draw[EdgeGraph](a)--(b);
        \draw[EdgeGraph](b)edge[bend left=40](c);
        \draw[EdgeGraph](b)edge[bend right=40](c);
    \end{tikzpicture}
\end{equation}
is in $\K \MG$ but is not in $\LP$.
\end{itemize}
\end{proposition}

%%%%%%%%%%%%%%%%%%%%%%%%%%%%%%%%%%%%%%%%%%%%%%%%%%%%%%%%%%%%%%%%%%%%%%%%%%%%%%%%%%%%%%%%%%%%
\section*{Concluding remarks}
We defined in this extended abstract a notion of graph insertion operad. In the complete
version~\cite{AGKT19} of this paper, we give an even more general definition of graph
insertion operads which also naturally extends to hypergraphs.

There are two main questions, with reciprocical goals, raised by this paper: the description
of the multigraphs generated contained in $\LP$ and the description of the generators of the
various operads defined here (as $\K\G_{orc}^{\bullet}$, $\K\G_c$, $\K\T$, {\em etc.}).
Another perspective for future work is to study appropriate examples of algebras on $\SP$
and~$\SP^!$.

\bibliographystyle{plain}
\bibliography{FPSAC}

\end{document}